\documentclass[11pt,letterpaper]{amsart}

\usepackage{header}
\numberwithin{equation}{subsection}
\begin{document}
\sloppy
\title{On Fractal Continuity Properties of Certain One-Dimensional Schrödinger Operators}
\author{Netanel Levi}
\thanks{The author was supported by NSF DMS-2052899, DMS-2155211, and
Simons 896624.}
\begin{abstract}
We construct examples of one-dimensional Schrödinger operators that illustrate the subtle nature of fractal continuity properties of spectral measures:
\begin{enumerate}
    \item Half-line operators whose spectral measures have packing dimension zero for all boundary conditions.
    \item A whole-line operator whose spectral measure has Hausdorff dimension one, while every half-line restriction (under any boundary condition) has spectral measure of Hausdorff dimension zero.
    \item For the same whole-line operator, we prove the existence of a Borel set to which the spectral measure assigns positive measure, but which has measure zero w.r.t.\ the spectral measure of the positive half-line restriction under every boundary condition.
\end{enumerate}
\end{abstract}

\maketitle
\section{Introduction}
	
In this work, we study one-dimensional Schrödinger operators acting on $\ell^2(\mathbb{N})$ and $\ell^2(\mathbb{Z})$. These operators are given by tridiagonal matrices that are either semi-infinite (in the half-line case) or doubly infinite (in the line case), with off-diagonal entries identically equal to $1$ and a real-valued sequence on the main diagonal. In both cases, the resulting operator is essentially self-adjoint \cite{Ber}. Our focus is on the continuity properties of the spectral measures associated with these operators.
	
In the half-line case, a Schrödinger operator $H$ gives rise to a one-parameter family of self-adjoint operators via rank-one perturbations. For every $\theta \in [0,\pi)$, we define  
	\[
	H_\theta = H - \tan\theta \langle\delta_1, \cdot \rangle \delta_1.
	\]  
	For each $\theta$, the vector $\delta_1$ is cyclic for $H_\theta$, and we denote its spectral measure by $\mu_\theta$. It is well known (see, e.g.\ \cite{Kato}) that the essential spectrum is invariant under finite-rank perturbations, namely there exists a closed set $\Sigma \subseteq \mathbb{R}$ such that for every $\theta \in [0,\pi)$,  
	\[
	\sigma_{\text{ess}}(H_\theta) = \Sigma.
	\]  
	Gordon \cite{Gor} and independently del Rio, Makarov, and Simon \cite{DMS} (see also \cite{DJMS}) proved that for a dense $G_\delta$ subset $\Theta \subseteq [0,\pi)$, the operators $H_\theta$ for $\theta \in \Theta$ have no point spectrum inside $\Sigma$. Since pure point spectrum represents, in a certain sense, the strongest form of spectral singularity, a natural question is whether this phenomenon extends to weaker forms of singularity, such as those characterized by local Hausdorff and packing dimensions (see Section \ref{section_prelim} for precise definitions). Examples of operators exhibiting fractional and even zero local Hausdorff dimension within the essential spectrum for all rank-one perturbations were constructed in \cite{DJLS,Z}. Our first result is the following:
	
	\begin{theorem}\label{main_thm_1}
		There exists a half-line Schrödinger operator $H$ such that $\Sigma = \sigma_{\text{ess}}(H) = [-2,2]$, and for every $\theta \in [0,\pi)$, the restriction of $\mu_\theta$ to $[-2,2]$ has packing dimension zero.
	\end{theorem}
	
	\begin{remark}
		\begin{enumerate}
			\item Since the packing dimension of a measure is always greater than or equal to its Hausdorff dimension, it follows that the family of operators in Theorem \ref{main_thm_1} also has Hausdorff dimension zero for all $\theta \in [0,\pi)$.
			\item There are known examples of operators which have zero packing dimension (even pure point spectrum) for all rank-one perturbations \cite{JK}. However, it was recently shown that for these examples, the spectrum does not contain an interval \cite{KPS}.
            \item The proof of Theorem \ref{main_thm_1} is based on a method of analyzing packing dimensional properties of the spectral measures based on asymptotic properties of solutions, which was recently developed in \cite{JLT}. 
		\end{enumerate}		
	\end{remark}

	In the second part of this work, we study certain line Schrödinger operators and the connections between these operators and their half-line restrictions. Given a Schrödinger operator $H$ acting on $\ell^2(\mathbb{Z})$, let $(H_\theta^\pm)_{\theta \in [0,\pi)}$ and $(\mu_\theta^\pm)_{\theta \in [0,\pi)}$ denote the families of rank-one perturbations and corresponding spectral measures arising from the restrictions of $H$ to the positive and negative half-lines (see Section \ref{section_line_ops} for precise definitions). Additionally, let $\mu$ be the sum of the spectral measures of $\delta_0$ and $\delta_1$ with respect to $H$.
	
	Using techniques developed in \cite{JL1} and further refined in \cite{DKL,KKL}, it was shown in \cite{DKL} that, roughly speaking, in terms of local Hausdorff continuity properties, the line operator is continuous at least as much as its half-line restrictions (see Theorem \ref{DKL_Cor}). On the other hand, in \cite{JL2}, the authors state that techniques from \cite{JL1,JL2} can be used to construct an example of a line operator with a (Hausdorff) one-dimensional spectrum, while the spectral measures of all of its half-line restrictions have Hausdorff dimension zero. Our second result confirms this claim:
	
	\begin{theorem}\label{main_thm_2}
		There exists a line Schrödinger operator $H$ whose essential spectrum is $\left[-2,2\right]$, and such that $\mu|_{\left[-2,2\right]}$ has Hausdorff dimension $1$, whereas for every $\theta \in [0,\pi)$, $\mu_\theta^\pm$ has Hausdorff dimension zero.
	\end{theorem}
	
	In the final part of this work, we show that every operator satisfying the properties listed in Theorem \ref{main_thm_2} also possesses another intriguing property:
	
	\begin{theorem}\label{main_thm_3}
		Let $H$ be a Schrödinger operator which satisfies the properties listed in Theorem \ref{main_thm_2}. Then there exists a Borel set $A \subseteq \mathbb{R}$ such that $\mu(A) > 0$, while for every $\theta \in [0,\pi)$, $\mu_\theta^+(A) = 0$.
	\end{theorem}
	
	\begin{remark}
		It is unknown if it is possible to obtain that $\mu_\theta^\pm=0$ for all $\theta\in\left[0,\pi\right)$, and for the moment this question seems to remain open.
	\end{remark}
	The rest of the paper is structured as follows. In Section \ref{section_prelim}, we present preliminaries in spectral theory, subordinacy theory, and the study of fractal continuity properties of Borel measures. In Section \ref{section_hl_ops}, we construct examples of half-line operators and prove Theorem \ref{main_thm_1}. In Section \ref{section_line_ops}, we discuss line operators and prove Theorems \ref{main_thm_2} and \ref{main_thm_3}.
	
	{\bf Acknowledgments} I would like to thank Jonathan Breuer, Svetlana Jitomirskaya and Yoram Last for useful discussions.

	\section{Preliminaries}\label{section_prelim}
	\subsection{Spectral measures of self-adjoint operators}
	Let $\mathcal{H}$ be a Hilbert space and let \mbox{$T:D\left(T\right)\to\mathcal{H}$} be a self-adjoint operator. Let $P$ be the spectral resolution of $T$. For every $\psi,\varphi\in\mathcal{H}$, the joint spectral measure of $\psi$ and $\varphi$ w.r.t.\ $T$ is given by
	\begin{center}
		$\mu_{\psi,\varphi}\left(A\right)=\langle P\left(A\right)\psi,\varphi\rangle,\,\,\,\,\,\,A\in\text{Borel}\left(\mathbb{R}\right)$.
	\end{center}
	$\mu_\psi$ is a finite Borel measure. If $\psi=\varphi$, then this is a positive measure and in that case we will denote this measure by $\mu_\psi$.
	
	It is well-known that given a bounded real-valued measurable function $f$, $f\left(T\right)$ is a well-defined self-adjoint operator $f\left(T\right):D\left(f\left(T\right)\right)\to\mathcal{H}$ (see, e.g.\ \cite{RS}).
	\begin{definition}
		A set $C\subseteq\mathcal{H}$ is called cyclic if
		\begin{center}
			$\mathcal{H}=\overline{\vspan\left\{f\left(\varphi\right):\varphi\in C,f\in C_\infty\left(\mathbb{R}\right)\right\}}$.
		\end{center}
	\end{definition}
	The following fact is well-known (see, e.g.\ \cite{DF}).
	\begin{lemma}\label{cyclicity_claim}
		Let $C\subseteq\mathcal{H}$ be a finite cyclic set and let $\mu\coloneqq\sum\limits_{\varphi\in C}\mu_\varphi$. Then for every $\psi\in\mathcal{H}$, $\mu_\psi\ll\mu$.
	\end{lemma}
	\begin{remark}
		Note that for every $\varphi,\psi\in\mathcal{H}$ and for every $A\in\text{Borel}\left(\mathbb{R}\right)$, we have
		\begin{center}
			$\langle P\left(A\right)\psi,\varphi\rangle\leq\langle P\left(a\right)\psi,\psi\rangle\langle P\left(A\right)\varphi,\varphi\rangle$
		\end{center}
		and so clearly we have $\mu_{\psi,\varphi}\ll\mu$ as well.
	\end{remark}
	\subsection{Hausdorff and packing measures}
	Throughout this subsection, we fix some $\alpha\in\left[0,1\right]$ and a finite positive Borel measure $\mu$. Our goal in this subsection is to present the definitions of the $\alpha$-dimensional Hausdorff and packing measures and discuss certain continuity properties of $\mu$ w.r.t.\ these measures.
	\subsubsection*{Definition of the $\alpha$-dimensional Hausdorff measure}
	Let $A\in\text{Borel}\left(\mathbb{R}\right)$. Given $\delta>0$, a $\delta$-cover of $A$ is a collection of open intervals $\left(I_j\right)_{j\in\mathbb{N}}$ such that $A\subseteq\underset{j\in\mathbb{N}}{\bigcup} I_j$, and for every $j\in\mathbb{N}$, $\left|I_j\right|\leq\delta$. Given $\delta>0$ and $A\in\text{Borel}\left(\mathbb{R}\right)$, let
	\begin{center}
		$h^\alpha\left(A\right)=\underset{\delta>0}{\lim}\,\inf\left\{\sum\limits_{j=1}^\infty\left|I_j\right|:\left(I_j\right)_{j\in\mathbb{N}}\text{ is a }\delta-\text{cover of }A\right\}$.
	\end{center}
	$h^\alpha$ is the $\alpha$-dimensional Hausdorff measure.
	
	Given $A\in\text{Borel}\left(\mathbb{R}\right)$, there exists a unique $\alpha\in\left[0,1\right]$ such that for all $\beta>\alpha$, $h^\beta\left(A\right)=0$ and for all $\beta<\alpha$, $h^\beta\left(A\right)=\infty$. We denote $\dim_H\left(A\right)\coloneqq\alpha$. We define the upper and lower Hausdorff dimension of the measure $\mu$ by
	\begin{center}
		$\dim_H^+\left(\mu\right)=\inf\left\{\dim_H\left(S\right):\mu\left(\mathbb{R}\setminus S\right)=0\right\}$,\\$\text{}$\\
		$\dim_H^-\left(\mu\right)=\inf\left\{\dim_H\left(S\right):\mu\left(S\right)>0\right\}$.
	\end{center}
	\subsubsection*{Definition of the $\alpha$-dimensional packing measure}
	Let $A\in\text{Borel}\left(\mathbb{R}\right)$. Given $\delta>0$, a $\delta$-packing of $A$ is a collection of disjoint closed intervals $\left(I_j\right)_{j\in\mathbb{N}}$ whose centers are all in $A$. Given $\delta>0$ and $A\in\text{Borel}\left(\mathbb{R}\right)$, let
	\begin{center}
		$p_0^\alpha\left(A\right)\coloneqq\underset{\delta\to 0}{\lim}\,\inf\left\{\sum\limits_{j=1}^\infty\left|I_j\right|:\left(I_j\right)_{j\in\mathbb{N}}\text{ is a }\delta-\text{packing of }A\right\}$.
	\end{center}
	The packing dimension $p^\alpha$ is given by
	\begin{center}
		$p^\alpha\left(A\right)=\inf\left\{\sum\limits_{j=1}^\infty p_0^\alpha\left(A_j\right):A\subseteq\underset{j\in\mathbb{N}}{\bigcup} A_j, A_j\in\text{Borel}\left(\mathbb{R}\right)\right\}$.
	\end{center}
	Given $A\in\text{Borel}\left(\mathbb{R}\right)$, there exists a unique $\alpha\in\left[0,1\right]$ such that for all $\beta>\alpha$, $p^\beta\left(A\right)=0$ and for all $\beta<\alpha$, $p^\beta\left(A\right)=\infty$. We denote $\dim_P\left(A\right)\coloneqq\alpha$. We define the upper and lower Hausdorff dimension of the measure $\mu$ by
	\begin{center}
		$\dim_P^+\left(\mu\right)=\inf\left\{\dim_P\left(S\right):\mu\left(\mathbb{R}\setminus S\right)=0\right\}$,\\$\text{}$\\
		$\dim_P^-\left(\mu\right)=\inf\left\{\dim_P\left(S\right):\mu\left(S\right)>0\right\}$.
	\end{center}
	We also define, for every $E\in\mathbb{R}$,
	\begin{center}
		$\gamma_\mu^-\left(E\right)=\underset{\varepsilon\to 0_+}{\liminf}\,\frac{\log\left(\mu\left(E-\varepsilon,E+\varepsilon\right)\right)}{\log\varepsilon}$,\\
		$\gamma_\mu^+\left(E\right)=\underset{\varepsilon\to 0_+}{\limsup}\,\frac{\log\left(\mu\left(E-\varepsilon,E+\varepsilon\right)\right)}{\log\varepsilon}$.
	\end{center}
	\begin{definition}
		Given $\alpha\in\left(0,1\right)$, we say that $\mu$ is $\alpha$-Hausdorff ($\alpha$-packing) singular if $\mu$ is supported on a set of $h^\alpha$ (of $p^\alpha$) measure zero. We say that $\mu$ is $\alpha$-Hausdorff ($\alpha$-packing) continuous if for every Borel set $A$, $h^\alpha\left(A\right)=0$ ($p^\alpha\left(A\right)=0$) implies that $\mu\left(A\right)=0$.
	\end{definition}
	Given $\alpha\in\left(0,1\right)$, let
	\begin{center}
		$\overline{D}_\mu^\alpha\left(E\right)\coloneqq\underset{\varepsilon\to0}{\limsup}\frac{\mu\left(E-\varepsilon,E+\varepsilon\right)}{\varepsilon^\alpha},$\\$\text{}$\\
		$\underline{D}_\mu^\alpha\left(E\right)\coloneqq\underset{\varepsilon\to0}{\liminf}\frac{\mu\left(E-\varepsilon,E+\varepsilon\right)}{\varepsilon^\alpha}$.
	\end{center}
	We will use the following results which can be found in the Appendix of \cite{GSB} (see also \cite{Cut}).
	\begin{lemma}\label{packing_derivative_infinity_lemma}
		For every $\eta>\gamma_\mu^+\left(E\right)$, $\underline{D}_\mu^\eta\left(E\right)=\infty$.
	\end{lemma}
	We will also use the following 
	\begin{proposition}\emph{\cite[Corollary 13]{CdO}}\label{prop_hp_decomposition}
		Let $\alpha\in\left(0,1\right)$. Denote
		\begin{center}
			$T_\infty^\alpha\coloneqq\left\{E\in\mathbb{R}:\overline{D}_\mu^\alpha\left(E\right)=\infty\right\}$,\\$\text{}$\\
			$U_\infty^\alpha\coloneqq\left\{E\in\mathbb{R}:\underline{D}_\mu^\alpha\left(E\right)=\infty\right\}$.
		\end{center}
		We have the following.
		\begin{enumerate}
			\item Denote
			\begin{center}
				$\mu_{\alphash}\coloneqq\mu\left(T_\infty^\alpha\cap\cdot\right),\,\,\mu_{\alphach}\coloneqq\mu\left(\mathbb{R}\setminus T_\infty^\alpha\cap\cdot\right)$. 
			\end{center}
			Then $\mu_{\alphash}$ is $\alpha$-Hausdorff singular and $\mu_{\alphach}$ is $\alpha$-Hausdorff continuous.
			\item Denote
			\begin{center}
				$\mu_{\alphasp}\coloneqq\mu\left(U_\infty^\alpha\cap\cdot\right),\,\,\mu_{\alphacp}\coloneqq\mu
				\left(\mathbb{R}\setminus U_\infty^\alpha\cap\cdot\right)$.
			\end{center}
			Then $\mu_{\alphasp}$ is $\alpha$-packing singular and $\mu_{\alphacp}$ is $\alpha$-packing continuous.
		\end{enumerate}
	\end{proposition}
	\begin{remark}
		Part $1$ of Proposition \ref{prop_hp_decomposition} was originally proved by Rogers and Taylor \cite{RT}.
	\end{remark}
	\subsubsection*{The Borel transform}
	Let $\mu$ be a finite Borel measure. The \textit{Borel transform} (also called Stieltjes transform) of $\mu$ is an analytic function which maps $\mathbb{C}_+\coloneqq\left\{z\in\mathbb{C}:\im z>0\right\}$ to itself, given by
	\begin{center}
		$m\left(z\right)=\int_\mathbb{R}\frac{d\mu\left(x\right)}{x-z}$.
	\end{center}
	The boundary behavior of $m$ is strongly connected to continuity properties of $\mu$. To that end, let
	\begin{center}
		$Q_\mu^\alpha\left(E\right)\coloneqq\underset{\varepsilon\to0}{\limsup}\,\varepsilon^{1-\alpha}\im m_\mu\left(E+i\varepsilon\right)$\\$\text{}$\\
		$R_\mu^\alpha\left(E\right)=\underset{\varepsilon\to0}{\limsup}\,\varepsilon^{1-\alpha}\left|m_\mu\left(E+i\varepsilon\right)\right|$.
	\end{center}
	We will use the following
	\begin{proposition}\emph{\cite[Theorem 3.1]{DJLS}}\label{DJLS_thm}
		For every finite Borel measure and $E\in\mathbb{R}$ and $\alpha\in\left(0,1\right)$, $Q_\mu^\alpha\left(E\right),R_\mu^\alpha\left(E\right)$ and $\overline{D}_\mu^\alpha\left(E\right)$ are either all infinite, all zero or all in $\left(0,\infty\right)$.
	\end{proposition}
	We will also use the following theorem of Kac:
	\begin{proposition}\emph{\cite{Kac}}\label{rnd_m_func_thm}
		Let $\mu,\nu$ be probability measures. Then for $\nu$-almost every $E\in\mathbb{R}$, we have
		\begin{center}
			$\frac{d\mu}{d\nu}\left(E\right)=\underset{\varepsilon\to0}{\lim}\frac{\im m_\mu\left(E+i\varepsilon\right)}{\im m_\nu\left(E+i\varepsilon\right)}$.
		\end{center}
	\end{proposition}
        The following proposition, proved in \cite{JLT}, provides a certain connection between the upper and lower local dimensions of $\mu$, given some information about the boundary behavior of its Borel transform.
	\begin{proposition}\emph{\cite[Theorem 1.3]{JLT}}\label{packing_dim_borel_transform_prop}
		Let $0\leq\eta<1$. Suppose that $\underset{\varepsilon\to 0}{\liminf}\,\varepsilon^{1-\eta}\im m\left(E+i\varepsilon\right)>0$. Then
		\begin{equation}
			\gamma_\mu^+\left(E\right)\leq\frac{\eta\left(2-\gamma_\mu^-\left(E\right)\right)}{2-\eta}.
		\end{equation}
		In particular, $\gamma_\mu^+\left(E\right)\leq\frac{2\eta}{2-\eta}$.
	\end{proposition}
	\subsection{Half-Line Operators}\label{section_hl_ops}
	In this section, we discuss half-line Schr\"{o}dinger operators, namely operators \mbox{$H:D\left(H\right)\subseteq\ell^2\left(\mathbb{N}\right)\to\ell^2\left(\mathbb{N}\right)$} of the form
	\begin{equation}\label{eq_hl_op}
		\left(H\psi\right)\left(n\right)=\begin{cases}
			\psi\left(n-1\right)+\psi\left(n+1\right)+V\left(n\right)\psi\left(n\right) & n\geq 2\\
			\psi\left(2\right)+V\left(1\right)\psi\left(1\right) & n=1
		\end{cases}
	\end{equation}
	where $V:\mathbb{N}\to\mathbb{R}$ is some sequence. Such operators are essentially self-adjoint \cite{Ber}, and it is not hard to see that $\delta_1$ is cyclic for $H$. We denote its spectral measure by $\mu$. We will also be interested in rank-one perturbations of $H$. To that end, for every $\theta\in\left[0,\pi\right)$, we define
	\begin{equation}\label{eq_rank_one_pert_hl}
		H_\theta=\begin{cases}
			H-\tan\theta\langle\delta_1,\cdot\rangle\delta_1 & \theta\neq\frac{\pi}{2}\\
			\widetilde{H} & \theta=\frac{\pi}{2}
		\end{cases},
	\end{equation}
	where $\widetilde{H}$ is defined by (\ref{eq_hl_op}) with the potential $\widetilde{V}$ given by shifting $V$ once to the left. For every $\theta\in\mathbb{R}$, $\delta_1$ is cyclic for $H_\theta$. We denote its spectral measure and respective Borel transform by $\mu_\theta$ and $m_\theta$.
	
	We will need the following lemma.
	\begin{lemma}\label{lemma_ess_spec}
		Let $\left(n_k\right)_{k=1}^{\infty}$ be strictly increasing sequence of natural numbers which satisfies \mbox{$\underset{k\to\infty}{\lim}\,n_k-n_{k-1}=\infty$}. Let $H$ be of the form (\ref{eq_hl_op}). Suppose that for every $n\notin\left\{n_k:k\in\mathbb{N}\right\}$, $V\left(n\right)=0$ and that $V\left(n_k\right)\underset{k\to\infty}{\longrightarrow}\infty$. Then $\sigma_{\text{ess}}\left(H\right)=\left[-2,2\right]$.
	\end{lemma}
	\begin{proof}
		Denote by $\Delta$ the free Laplacian, namely $\left(\Delta\psi\right)\left(n\right)=\psi\left(n-1\right)+\psi\left(n+1\right)$ and recall that $\sigma\left(\Delta\right)=\sigma_{\text{ess}}\left(\Delta\right)=\left[-2,2\right]$. Clearly, every Weyl sequence for the Laplacian generates a Weyl sequence for $H$ and so $\left[-2,2\right]\subseteq\sigma_{\text{ess}}\left(H\right)$. On the other hand, let $\left(\psi_k\right)_{k=1}^\infty$ be any sequence of orthonormal vectors. Note that as a multiplication operator, $V$ has an orthonormal basis of eigenvectors $\left(\delta_n\right)_{n=1}^\infty$ with isolated eigenvalues $\left(\lambda_n\right)_{n=1}^{\infty}$ such that the only eigenvalue of multiplicity greater than one is $0$. This implies that given $E\notin\left[-2,2\right]$, there exists $N\in\mathbb{N}$ such that for every $n>N$, $\lambda_n\notin\left(E-2,E+2\right)$. Denote $\dist\left(E,\left\{\lambda_n:n>N\right\}\right)=2+\eta$ for some $\eta>0$. In addition, since $\left(\psi_k\right)_{k=1}^\infty$ is an orthonormal sequence, we have that for every $\varepsilon>0$ for large enough $K$, for every $k>K$ and for every $1\leq n\leq N$,
		\begin{equation}
			\left|\psi_k\left(n\right)\right|\leq\frac{\varepsilon}{N}.
		\end{equation}
		Finally we obtain
		\begin{center}
			$\|H\psi_k-E\psi_k\|\geq\left\|\left(V-E\right)\psi_k\right\|-\left\|\Delta\psi_k\right\|=\left\|\sum\limits_{n\leq N}\psi_k\left(N\right)\left(E-\lambda_n\right)\delta_n+\sum\limits_{n>N}\psi_k\left(n\right)\left(E-\lambda_n\right)\delta_n\right\|-\left\|\Delta\psi_k\right\|\geq\left\|\sum\limits_{n>N}\psi_k\left(n\right)\left(E-\lambda_n\right)\delta_n\right\|-\left(\left\|\sum\limits_{n\leq N}\psi_k\left(n\right)\left(E-\lambda_n\right)\delta_n\right\|+\left\|\Delta\psi_k\right\|\right)>\left(2+\frac{\eta}{2}\right)\|\psi_k\|-\left(2-\varepsilon\right)\|\psi_k\|$.
		\end{center}
		Taking $\varepsilon$ to be small enough, we see that $\left(\psi_k\right)_{k=1}^\infty$ cannot be a Weyl sequence for $E$ and so \mbox{$E\notin\sigma_{\text{ess}}\left(H\right)$}, as required.
	\end{proof}
	\begin{remark}
		By standard finite-rank perturbation arguments, Lemma \ref{lemma_ess_spec} implies that for every \mbox{$\theta\in\left[0,\pi\right)$}, $\sigma_{\text{ess}}\left(H_\theta\right)=\left[-2,2\right]$.
	\end{remark}
	\subsubsection*{Transfer matrices}
	Let us introduce the notion of transfer matrices associated with $H$ and some of the spectral theory connected with them. Fix $E\in\mathbb{R}$ and consider $u:\mathbb{N}\cup\left\{0\right\}\to\mathbb{R}$ which satisfies
	\begin{equation}\label{eq_ev_hl}
		u\left(n-1\right)+u\left(n+1\right)+V\left(n\right)u\left(n\right)=Eu\left(n\right),\,\,\,\,\,n\in\mathbb{N}.
	\end{equation}
	For every $n\in\mathbb{N}$, we denote
	\begin{center}
		$T_n\left(E\right)\coloneqq\left(\begin{matrix}
			E-V\left(n\right) & -1\\
			1 & 0
		\end{matrix}\right)$
	\end{center}
	and for $k,m\in\mathbb{N}$ such that $k<m$, we denote
	\begin{center}
		$\Phi_{k,m}\coloneqq T_m\left(E\right)T_{m-1}\left(E\right)\cdots T_{k}\left(E\right)T_{k-1}\left(E\right)$,\\
		$\Phi_{m}\left(E\right)\coloneqq\Phi_{1,m}\left(E\right)$.
	\end{center}
	It is well-known (see, e.g.\ \cite{DF2}) that
	\begin{center}
		$\forall n\in\mathbb{N} ,\,\,\,T_n\left(E\right)\left(\begin{matrix} u\left(n\right)\\ u\left(n-1\right)\end{matrix}\right)=\left(\begin{matrix} u\left(n+1\right)\\u\left(n\right)\end{matrix}\right)$
	\end{center}
	and so for every $k<m\in\mathbb{N}$,
	\begin{center}
		$\Phi_{k,m}\left(E\right)\left(\begin{matrix}u\left(k\right)\\u\left(k-1\right)\end{matrix}\right)=\left(\begin{matrix}u\left(m+1\right)\\u\left(m\right)\end{matrix}\right)$.
	\end{center}
	In particular, for every $n\in\mathbb{N}$,
	\begin{center}
		$\Phi_n\left(E\right)\left(\begin{matrix}u\left(1\right)\\u\left(0\right)\end{matrix}\right)=\left(\begin{matrix}u\left(n+1\right)\\u\left(n\right)\end{matrix}\right)$.
	\end{center}
	\begin{claim}
		For every $k<m\in\mathbb{N}$ and for every $E\in\mathbb{R}$, $\Phi_m,\Phi_{k,m}$ and $T_m$ are all invertible. In addition, for every closed interval $I\subseteq\mathbb{R}$, the functions $\Phi_m,\Phi_{k,m},T_m:I\to M_2\left(\mathbb{R}\right)$ are all continuous.
	\end{claim}
	
	\begin{proof}
		This is immediate from the definitions of $\Phi_m\left(E\right),\Phi_{k,m}\left(E\right)$ and $T_m\left(E\right)$.
	\end{proof}
	For every $E\in\mathbb{R}$, let $\overline{\gamma}\left(E\right)$ be the upper Lyapunov exponent:
	\begin{center}
		$\overline{\gamma}\left(E\right)=\underset{n\to\infty}{\limsup}\frac{1}{n}\ln\|\Phi_n\left(E\right)\|$.
	\end{center}
	We will need the following theorem from \cite{JL1}:
	\begin{proposition}\label{JLTHM}
		Suppose that $\overline{\gamma}\left(E\right)>0$ for every $E$ in some Borel set $A$. Then for every $\theta\in\left[0,\pi\right)$, the restriction $\mu_\theta\left(A\cap\cdot\right)$ is zero-dimensional.
	\end{proposition}
	\subsubsection*{Subordinacy theory}
	Fix $E\in\mathbb{R}$ and let $\text{Sol}\left(E\right)$ be the set of all $u:\mathbb{N}\cup\left\{0\right\}\to\mathbb{R}$ which satisfy (\ref{eq_ev_hl}) for all $n\geq 1$, and in addition
	\begin{equation}\label{eq_norm_sol_hl}
		\left|u\left(0\right)\right|^2+\left|u\left(1\right)\right|^2=1.
	\end{equation}
	For every $\theta\in\left[0,\pi\right)$, we denote by $u_\theta,v_\theta$ the unique $u,v\in\text{Sol}\left(E\right)$ which satisfy
	\begin{center}
		$\left(\begin{matrix}
			u_\theta\left(1\right) & v_\theta\left(1\right)\\u_\theta\left(0\right) & v_\theta\left(0\right)
		\end{matrix}\right)=\left(\begin{matrix}
		\cos\theta & \sin\theta\\-\sin\theta & \cos\theta
	\end{matrix}\right)$.
	\end{center}
	We will use the well-known fact that the Wronskian is constant (see e.g.\ \cite{JL1}).
	\begin{equation}\label{eq_const_wron}
		u_\theta\left(n+1\right)v_\theta\left(n\right)-u_\theta\left(n\right)v_\theta\left(n+1\right)=1,\,\,\,\,\,\,\,\,n\in\mathbb{N}.
	\end{equation}
	For every $u:\mathbb{N}\to\mathbb{C}$ and $L>0$, we define
	\begin{center}
		$\|u\|_L=\left(\sum\limits_{k=1}^{\floor{L}}\left|u\left(k\right)\right|^2+\left(L-\floor{L}\right)\left|u\left(\floor{L}+1\right)\right|^2\right)^{\frac{1}{2}}$.
	\end{center}
	Fix $\theta\in\left[0,\pi\right)$. For every $L>0$, we let
	\begin{equation}\label{eq_a_b_omega_def}
		\begin{aligned}
			a\left(L\right) & = \|v_{\theta}\|_L,\\
			b\left(L\right) & = \|u_\theta\|_L,\\
			\omega\left(L\right) & = \left(\underset{\eta\in\left[0,\pi\right)}{\max}\,\|u_\eta\|_L\right)\cdot\left(\underset{\eta\in\left[0,\pi\right)}{\min}\,\|u_\eta\|_L\right).
		\end{aligned}
	\end{equation}
	\begin{definition}
		Given $E\in\mathbb{R}$ and $\theta\in\left[0,\pi\right)$, we say that $u_\theta$ is \textit{subordinate} if
		\begin{center}
			$\underset{L\to\infty}{\lim}\frac{\|u_\theta\|_L}{\|v_\theta\|_L}=0$.
		\end{center}
	\end{definition}
	It is not hard to see that $\omega\left(L\right)$ is monotone increasing and goes to infinity as $L\to\infty$. Thus, for every $\varepsilon>0$ there exists a unique $L\left(\varepsilon\right)$ such that $\omega\left(L\left(\varepsilon\right)\right)=\frac{1}{\varepsilon}$. The function $\varepsilon\to L\left(\varepsilon\right)$ is monotone decreasing and goes to $\infty$ as $\varepsilon\to 0$. We will use the following
	\begin{proposition}\emph{\cite[Theorem 2.3]{KKL}}\label{prop_JL}
		For every $\theta\in\left[0,\pi\right)$, $E\in\mathbb{R}$ and $\varepsilon>0$,
		\begin{equation}\label{JL_ineq}
			\frac{2-\sqrt{3}}{\left|m_{\mu_\theta}\left(E+i\varepsilon\right)\right|}<\frac	{\left\|u_\theta\right\|_{L\left(\varepsilon\right)}}{\left\|v_\theta\right\|_{L\left(\varepsilon\right)}}<\frac{2+\sqrt{3}}{\left|m_{\mu_\theta}\left(E+i\varepsilon\right)\right|}.
		\end{equation}
	\end{proposition}
	\begin{remark}
		\begin{enumerate}
			\item Proposition \ref{prop_JL} was originally proved in \cite{JL1} with different constant and with $L\left(\varepsilon\right)$ defined by requiring $\|u_\theta\|_{L\left(\varepsilon\right)}\cdot\|v_\theta\|_{L\left(\varepsilon\right)}=\frac{1}{2\varepsilon}$.
			\item Proposition \ref{prop_JL} implies the original version of subordinacy theory, proved by Gilbert and Pearson \cite{GP} in the continuum case and by Khan and Pearson \cite{KP} in the discrete case, which says that the singular part of $\mu_\theta$ is supported on the set of $E\in\mathbb{R}$ for which $u_\theta$ is subordinate.
		\end{enumerate}
	\end{remark}
	The following is proved in \cite{DT} (see also \cite{KKL}).
	\begin{proposition}\label{DT_prop}
		For every $\varepsilon>0$ and for every $L>0$,
		\begin{equation}
			\im m_\theta\left(E+i\varepsilon\right)\geq\frac{\varepsilon\omega^2\left(L\right)}{b\left(L\right)\left(1+\varepsilon\omega\left(L\right)\right)^2}.
		\end{equation}
		In particular, for $L=L\left(\varepsilon\right)$,
		\begin{equation}
			\im m_\theta\left(E+i\varepsilon\right)\geq\frac{1}{4\varepsilon b\left(L\right)}.
		\end{equation}
	\end{proposition}
	We will also need
	\begin{proposition}\emph{\cite[Theorem 3.10]{LS}}\label{prop_gef_bd}
		For every $\theta\in\left[0,\pi\right)$, for $\mu_\theta$-almost every $E\in\mathbb{R}$, the solution $u_\theta$ satisfies
		\begin{equation}\label{eq_gen_ef_hl}
			\|u_\theta\|_L\leq C\left(E\right)L^{\frac{1}{2}}\ln L.
		\end{equation}
	\end{proposition}
	\section{Certain examples of half-line operators}
    In this section, we prove Theorem \ref{main_thm_1} and construct certain half-line operators which will be used in the proof of Theorems \ref{main_thm_2} and \ref{main_thm_3}
    \subsection{Operators with packing-dimension zero for every rank-one perturbation}
	In this subsection we prove Theorem \ref{main_thm_1}. Namely, we will define a potential $V:\mathbb{N}\to\mathbb{R}$ such that the resulting family of Schr\"{o}dinger operators $\left(H_\theta\right)_{\theta\in\left[0,\pi\right)}$ satisfies the following property: For every $\theta\in\mathbb{R}$, $\dim_P\left(\mu_\theta\right)=0$.
	
	We begin the construction by setting $V\left(n\right)=0$ for every $n\notin\left\{k^2:k\in\mathbb{N}\right\}$. We will define $V\left(k^2\right)$ inductively. Fix $k\in\mathbb{N}$ and suppose that $V\left(1\right),\ldots,V\left(k^2-1\right)$ are all defined. Denote $n=k^2$, and let
	\begin{center}
		$C_n\coloneqq\underset{E\in\left[-2,2\right]}{\min}\frac{1}{\|\Phi_{n-1}\left(E\right)\|}>0$.
	\end{center}
	\begin{lemma}
		There exists $M>0$ such that if $V\left(n\right)>M$, then for every $E\in\left[-2,2\right]$,
		\begin{equation}\label{eq_op_norm_Tn}
			\|T_n\left(E\right)\|\cdot C_n>\left(k+1\right)^{k+1}
		\end{equation}
	\end{lemma}
	\begin{proof}
		Note that since all norms on $M_2\left(\mathbb{R}\right)$ are equivalent, it suffices to show that (\ref{eq_op_norm_Tn}) holds when replacing $\|\cdot\|$ by $\|\cdot\|_\infty$. Now, note that there exists $N>0$ such that for every $E\in\left[-2,2\right]$ and for every $x>N$,
		\begin{center}
			$\left\|\left(\begin{matrix}
				E-x & -1\\
				1 & 0
			\end{matrix}\right)\right\|_\infty=\left|E-x\right|\geq \frac{x}{2}$.
		\end{center}
		Now, one can take $M=\max\left\{2\left(k+1\right)^{k+1}\cdot\frac{1}{C_n},N\right\}$ and the result follows.
	\end{proof}
	Given $E\in\mathbb{R}$, let $\text{Sol}\left(E\right)$ be the set of all functions $u:\mathbb{N}\cup\left\{0\right\}\to\mathbb{R}$ which satisfy (\ref{eq_ev_hl}) for all $n\in\mathbb{N}$, and in addition
	We have the following
	\begin{corollary}\label{cor_L_norms}
		For every $E\in\left[-2,2\right]$ and for every $4\leq L\in\mathbb{N}$, $\underset{u\in\text{Sol}\left(E\right)}{\max}\,\|u\|_L\geq L^L$.
	\end{corollary}
	\begin{proof}
		Let $L\in\mathbb{N}$ and let $k\in\mathbb{N}$ be such that $k^2\leq L<\left(k+1\right)^2$. Denote $n=k^2$. Let $y\in\mathbb{R}^2$, $\|y\|=1$ such that $\|T_{n}\left(E\right)y\|=\|T_n\left(E\right)\|$. In addition, let $x\in\mathbb{R}^2$ such that $\Phi_{n-1}\left(E\right)x=y$ and consider $\widetilde{x}=\frac{x}{\|x\|}$. Let $u\in\text{Sol}\left(E\right)$ be the solution which satisfies
		\begin{center}
			$\left(\begin{matrix}
				u\left(1\right)\\
				u\left(0\right)
			\end{matrix}\right)=\widetilde{x}$.
		\end{center}
		Then we have
		\begin{equation}\label{eq_lower_bd_on_uL}
			\|u\|_L\geq\|u\|_n\geq\left(u\left(n-1\right)^2+u\left(n\right)^2\right)^{\frac{1}{2}}=\|\Phi_n\left(E\right)\widetilde{x}\|=\frac{1}{\|x\|}\|T_n\left(E\right)\cdot\Phi_{n-1}\left(E\right)x\|=\frac{1}{\|x\|}\|T_n\left(E\right)y\|.
		\end{equation}
		Now, note that since $x=\Phi_n-1\left(E\right)y$, we have
		\begin{equation}\label{eq_connection_between_x_and_y_norms}
			\|x\|=\|\Phi_{n-1}^{-1}\left(E\right)y\|\leq\|\Phi_{n-1}^{-1}\left(E\right)\|=\|\Phi_{n-1}\left(E\right)\|.
		\end{equation}
		Plugging (\ref{eq_connection_between_x_and_y_norms}) in (\ref{eq_lower_bd_on_uL}) and using the fact that $\|T_n\left(E\right)y\|=\|T_n\left(E\right)\|$, we obtain
		\begin{center}
			$\frac{1}{\|x\|}\|T_n\left(E\right)y\|\geq\frac{1}{\|\Phi_{n-1}\left(E\right)\|}\|T_n\left(E\right)\|\geq C_n\|T_n\left(E\right)\|\geq\left(k+1\right)^{k+1}$.
		\end{center}
		Finally, clearly $\left(k+1\right)^{k+1}\geq L^L$, and the result follows.
	\end{proof}
	We are now ready to prove Theorem \ref{main_thm_1}.
	\begin{proof}[Proof of Theorem \ref{main_thm_1}]
		We will show that the family of rank-one perturbations $\left(H_\theta\right)_{\theta\in\left[0,\pi\right)}$ which corresponds with $H$ defined at the beginning of this section satisfies the desired property. Namely, for every $\theta\in\left[0,\pi\right)$, $\dim_P\left(\mu_\theta|_{\left[-2,2\right]}\right)=0$.
		
		Recall the definitions of $a,b,\omega$ (\ref{eq_a_b_omega_def}). By Proposition \ref{packing_dim_borel_transform_prop}, it is enough to show that for every $t\in\left(0,1\right)$, for $\mu_\theta$-almost every $E\in\mathbb{R}$ there exist $\varepsilon_0>0$ such that for every $\varepsilon<\varepsilon_0$,
		\begin{equation}\label{eq_im_m_epsilon_minus_t}
			\im m_\theta\left(E+i\varepsilon\right)\geq\varepsilon^{-t}.
		\end{equation}
		Recall that for $\mu_\theta$-almost every $E\in\mathbb{R}$, the solution $u_\theta$ satisfies (\ref{eq_gen_ef_hl}). Thus we have
		\begin{equation}\label{eq_bd_b_L}
			b\left(L\right)\leq C\left(E\right)L^{1+\varepsilon}.
		\end{equation}
		In addition, note that for every $\varepsilon>0$, by Corollary \ref{cor_L_norms} we have
		\begin{center}
			$\frac{1}{\varepsilon}=\omega\left(L\left(\varepsilon\right)\right)\geq L\left(\varepsilon\right)^{L\left(\varepsilon\right)}$.
		\end{center}
		For a fixed $t\in\left(0,1\right)$, for small enough $\varepsilon$ we have $\frac{1}{L\left(\varepsilon\right)}<1-t$ and so we obtain
		\begin{equation}\label{eq_bd_1_epsilon}
			\frac{1}{\varepsilon^{1-t}}\geq\frac{1}{\varepsilon^{\frac{1}{L\left(\varepsilon\right)}}}\geq L\left(\varepsilon\right).
		\end{equation}
		By Proposition \ref{DT_prop}, (\ref{eq_bd_b_L}) and (\ref{eq_bd_1_epsilon}), we obtain
		\begin{equation}
			\im m_\theta\left(E+i\varepsilon\right)\geq\frac{1}{4\varepsilon b\left(L\left(\varepsilon\right)\right)}\geq\frac{1}{4C\left(E\right)\varepsilon L\left(\varepsilon\right)^{1+\delta}}\geq\frac{1}{4C\left(E\right)\varepsilon^{1-\left(1-t\right)^\delta}}.
		\end{equation}
		Letting $\delta\to 0$, we obtain (\ref{eq_im_m_epsilon_minus_t}) and so Theorem \ref{main_thm_1} is proved.
	\end{proof}
	\subsection{Operators with Hausdorff-dimension zero for every rank-one perturbation}
	In this section, we will define a potential $V:\mathbb{N}\to\mathbb{R}$ such that the resulting family of Schr\"{o}dinger operators $\left(H_\theta\right)_{\theta\in\left[0,\pi\right)}$ will have zero-dimensional spectral measures for all $\theta$. Given that the family constructed in the previous section has zero packing dimension and that in general the packing dimension is greater than or equal to the Hausdorff dimension, this subsection might seem redundant. However, as we will show, by taking potentials which are extremely sparse, one can verify that the Borel transforms of the corresponding spectral measures will only be large occasionally as $\varepsilon$ goes to zero. This construction will come in handy in the next section, where we will discuss some line operators given by pasting two half-line operators defined using ideas we present here.
	
	Let $\left(L_n\right)_{n=1}^\infty$ be any sequence of natural numbers which satisfies the assumptions of Lemma \ref{lemma_ess_spec}. Let $V:\mathbb{N}\to\mathbb{R}_{>0}$ be any sequence of positive numbers which vanishes on $\mathbb{N}\setminus\left\{L_k:k\in\mathbb{N}\right\}$, and in addition
	\begin{equation}\label{eq_potential_growth}
		\ln\left(V\left(L_n\right)\right)-\sum\limits_{k=1}^{n-1}\ln\left(V\left(L_k\right)\right)\geq L_n+1+n^2
	\end{equation}
	Let $H:\ell^2\left(\mathbb{N}\right)\to\ell^2\left(\mathbb{N}\right)$ be defined by $H=\Delta+V$.
	\begin{theorem}\label{thm_zero_hausdorff_dim}
		For every $\theta\in\left[0,\pi\right)$, $\sigma_{\text{ess}}\left(H_\theta\right)=\left[-2,2\right]$,  and the restriction $\mu_\theta\left(\left[-2,2\right]\cap\cdot\right)$ is zero-dimensional.
	\end{theorem}
	\begin{proof}
		The fact that the essential spectrum of $H_\theta$ consists of the interval $\left[-2,2\right]$ follows from Lemma \ref{lemma_ess_spec}. For the second part, we follow along the same lines of the proof of \cite[Theorem 1.3]{JL1} with slight modifications. Let $I=\left[a,b\right]\subseteq\left(-2,2\right)$ be an arbitrary closed interval. To establish the result, it suffices to show that $\mu\left(I\cap\cdot\right)$ has the desired properties. Let $E\in I$. Note that $\det\left(\Phi_{k,m}\left(E\right)\right)=1$ and so $\|\Phi_{k,m}^{-1}\left(E\right)\|=\|\Phi_{k,m}\left(E\right)\|$. In addition, for every $n\in\mathbb{N}$ and $L_n\leq k\leq m<L_{n+1}$, $\Phi_{k,m}\left(E\right)$ is the same as the corresponding transfer matrix for the free Laplacian, namely the half-line Schr\"{o}dinger operator which is defined by setting $b\equiv 0$. In particular, there is a constant $C_I$ such that $1\leq\|\Phi_{k,m}\left(E\right)\|<C_I$. Moreover, for every $n\in\mathbb{N}$, we have
		\begin{equation}
			\Phi_{L_n-1,L_n}\left(E\right)=T_{L_n}\left(E\right)=\left(\begin{matrix}
				E-V\left(L_n\right) & -1\\
				1 & 0
			\end{matrix}\right)
		\end{equation}
		and so
		\begin{equation}\label{norm_bounds}
			\max\left(1,V\left(L_n\right)-2\right)\leq\|T_{L_n}\left(E\right)\|\leq V\left(L_n\right)+3.
		\end{equation}
		Now, let $m\in\mathbb{N}$ and let $n\in\mathbb{N}$ such that $L_n\leq m<L_{n+1}$. Then, we have
		\begin{equation}
			\Phi_m\left(E\right)=\Phi_{L_n,m}\left(E\right)T_{L_n}\left(E\right)\Phi_{L_{n-1},L_n-1}\left(E\right)T_{L_{n-1}}\left(E\right)\cdots\Phi_{L_1,L_2-1}\left(E\right)T_{L_1}\left(E\right)\Phi_{L_1-1}\left(E\right)
		\end{equation}
		and so, using the inequalities $\|AB\|\geq\|A\|\frac{1}{\|B^{-1}\|}$, $\|AB\|\geq\|B\|\frac{1}{\|A^{-1}\|}$, the fact that \mbox{$\|\Phi_{k,m}\left(E\right)\|=\|\Phi_{k,m}^{-1}\left(E\right)\|$} and (\ref{norm_bounds}), we get
		\begin{equation} \label{eq1}
			\begin{split}
				\|\Phi_m\left(E\right)\|\geq\left(C_I^{n+1}\prod\limits_{k=1}^{n-1}\left(V\left(L_k\right)+3\right)\right)^{-1}\left(V\left(L_n\right)-2\right) & \geq\left(C_I^{n+1}\prod\limits_{k=1}^{n-1}\left(2V\left(L_k\right)\right)\right)^{-1}\left(\frac{V\left(L_n\right)}{2}\right) \\
				& \geq C^{-n}\left(\prod\limits_{k=1}^{n-1}V\left(L_k\right)\right)^{-1}V\left(L_n\right)
			\end{split}
		\end{equation}
		for some $C>0$, depending only on $I$. Now, for every $E\in I$, for $m=L_n$, we have
		\begin{center}
			$\frac{1}{m}\ln\|\Phi_m\left(E\right)\|\geq\frac{1}{L_n}\ln\left(C^{-n}\prod\limits_{k=1}^{n-1}V\left(L_k\right)V\left(L_n\right)\right)=\frac{1}{L_n}\left(-n\ln C-\sum\limits_{k=1}^{n-1}\ln\left(V\left(L_k\right)\right)+\ln\left(V\left(L_n\right)\right)\right)\geq\frac{1}{L_n}\left(-n\ln C+L_n+1+n\ln C\right)>1$,
		\end{center}
		where the last inequality is true for $n$ large enough by the properties of $V$. This implies that \mbox{$\overline{\gamma}\left(E\right)=\underset{n\to\infty}{\limsup}\frac{1}{n}\ln\|\Phi_n\left(E\right)\|\geq 1$} and so by Proposition \ref{JLTHM}, for every $\theta\in\left[0,\pi\right)$, $\mu_\theta\left(I\cap\cdot\right)$ is zero-dimensional.
	\end{proof}
	
	Theorem \ref{thm_zero_hausdorff_dim} shows that no matter how sparse a potential $V$ is, if the values it takes on its support are sufficiently large then the Borel transforms of the spectral measures will occasionally take large values close to the boundary. The next two results contrast this by saying that for very sparse potentials, one can also find points near the boundary for which the Borel transforms will (uniformly) take relatively small values.
	\begin{proposition}\label{bdd_free_prop}
		Let $H=\Delta+V$ be a Schr\"{o}dinger operator. Suppose that there exists $N\in\mathbb{N}$ such that for every $k\geq N$, $V\left(k\right)=0$. Then for every $0<\gamma_1<1$ and $\gamma_2>1$ there exist $L_0\in\mathbb{N}$ and $C_1,C_2>0$ such that for every $E\in\left[-2,2\right]$, every $u\in\text{Sol}\left(E\right)$ and every $L>L_0$,
		\begin{equation}\label{bdd_sol_eq}
			C_1L^{\gamma_1}\leq\|u\|_L\leq C_2L^{\gamma_2}.
		\end{equation}
		Furthermore, if $\widetilde{V}$ is a potential which satisfies $\widetilde{V}|_{\left[1,N+M\right]}=V|_{\left[1,N+M\right]}$ for some $M\in\mathbb{N}$, then for every $u\in\text{Sol}\left(E\right)$ and $L_0\leq L\leq L_0+M$, (\ref{bdd_sol_eq}) holds.
	\end{proposition}
	\begin{proof}
		For every $k_1,k_2\in\mathbb{N}$ such that $k_1\leq k_2$ and $u:\mathbb{N}\to\mathbb{C}$, we denote
		\begin{center}
			$\|u\|_{k_1,k_2}=\left(\sum\limits_{j=k_1}^{k_2}\left|u\left(j\right)\right|^2\right)^{\frac{1}{2}}$.
		\end{center}
		For every $k\in\left\{1,\ldots,N-1\right\}$, $E\to\|\Phi_{1,k}\left(E\right)\|$ is a positive continuous function on $\left[-2,2\right]$ and so it obtains minimal and maximal values. Along with the fact that $\det\Phi_{1,N}\left(E\right)=1$ for every $E\in\mathbb{R}$, we conclude that there exist constants $m,M>0$ such that for every $E\in\left[-2,2\right]$, every $u\in\text{Sol}\left(E\right)$ satisfies
		\begin{equation}
			m\leq\|u\|_{N-1}\leq M
		\end{equation}
		and 
		\begin{equation}\label{psi_u_bdd}
			m<\underset{\coloneqq\psi\left(u\right)}{\underbrace{\left\|\left(\begin{matrix} u\left(N-1\right)\\u\left(N\right)\end{matrix}\right)\right\|}}<M.
		\end{equation}
		In addition, $\frac{u}{\psi\left(u\right)}|_{k\geq N}$ is a normalized solution of (\ref{eq_ev_hl}) with $V\equiv 0$. It is well-known that in that case, (\ref{bdd_sol_eq}) holds and so there exist $L_1\in\mathbb{N}$ and $\widetilde{C}_1,\widetilde{C}_2>0$ such that for every $L>L_1$,
		\begin{equation}				\widetilde{C}_1\left(L-N\right)^{\gamma_1}\leq\left\|\frac{u}{\psi\left(u\right)}\right\|_{N,L}\leq \widetilde{C}_2\left(L-N\right)^{\gamma_2}.
		\end{equation}
		Plugging in (\ref{psi_u_bdd}) and denoting $\overline{C}_1=\widetilde{C}_1\psi\left(u\right)$, $\overline{C}_2=\widetilde{C}_2\psi\left(u\right)$, we get
		\begin{equation}
			\overline{C}_1\left(L-N\right)^{\gamma_1}\leq\|u\|_{N,L}\leq \overline{C}_2\left(L-N\right)^{\gamma_2}.
		\end{equation}
		In addition, there exists $L_2\in\mathbb{N}$ such that for every $L>L_2$, $L^{\gamma_2}\geq M$ and $\left(L-N\right)^{\gamma_1}\geq\frac{1}{2}L^{\gamma_1}$. Now, for every $L>L_0\coloneqq\max\left\{L_1,L_2\right\}$ we have
		\begin{center}
			$\|u\|_{L}\leq\|u\|_{N-1}+\|u\|_{N,L}\leq M+\left(L-N\right)^{\gamma_2}\leq C_2L^{\gamma_2}$
		\end{center}
		where $C_2=\overline{C}_2+1$. In addition,
		\begin{center}
			$\|u\|_L\geq\|u\|_{N,L}\geq\widetilde{C}_1\left(L-N\right)^{\gamma_1}\geq C_1L^{\gamma_1}$
		\end{center}
		where $C_1=\frac{\overline{C}_1}{2}$, as required. Now,	for every $m\in\mathbb{N}$, the first $m$ values of a normalized solution $u$ to (\ref{eq_ev_hl}) depend only on the $m$ first entries of the potential, and the result follows.
	\end{proof}
	\begin{corollary}\label{bdd_m_cor}
		Given $N\in\mathbb{N}$ and $a_1,\ldots,a_N\in\mathbb{R}$, for every $\alpha\in\left(0,1\right)$ there exists $\varepsilon>0$ such that for every $\eta\in\left(0,\varepsilon\right)$ there exists $K>N$ such that for every potential $V$ which satisfies
		\begin{enumerate}
			\item $V\left(i\right)=a_i$ for every $1\leq i\leq N$, and
			\item $V\left(i\right)=0$ for every $N+1\leq i\leq K$,
		\end{enumerate}
		for every $E\in\left[-2,2\right]$, $\theta\in\left[0,\pi\right)$ and $\delta\in\left(\eta,\varepsilon\right)$ we have
		\begin{equation}
			\delta^\alpha\left|m_\theta\left(E+i\delta\right)\right|\leq 1.
		\end{equation}
	\end{corollary}
	\begin{proof}
		This is a direct consequence of Propositions \ref{prop_JL} and \ref{bdd_free_prop}.
	\end{proof}
	\section{Line Operators}\label{section_line_ops}
	\subsection{Some general theory}
	Let $H:\ell^2\left(\mathbb{Z}\right)\to\ell^2\left(\mathbb{Z}\right)$ be a Schr\"{o}dinger operator. The pair $\left\{\delta_0,\delta_1\right\}$ is cyclic for $H$, and so we will study the measure $\mu=\mu_{\delta_0}+\mu_{\delta_1}$, the sum of the spectral measures of $\delta_0$ and $\delta_1$. Throughout, we will denote $\mathbb{Z}_+\coloneqq\left\{1,2,\ldots\right\}$ and $\mathbb{Z}_-\coloneqq\left\{0,-1,-2,\ldots\right\}$ and for every set $A\subseteq\mathbb{Z}$, we denote by $P_A$ the orthogonal projection of $\ell^2\left(\mathbb{Z}\right)$ onto $\ell^2\left(A\right)$.
	 As in the half-line case, we will study solutions to the eigenvalue equation
	 \begin{equation}\label{eq_ev_line}
	 	u\left(n-1\right)+u\left(n+1\right)+V\left(n\right)u\left(n\right)=Eu\left(n\right),\,\,\,\,\,\,n\in\mathbb{Z}.
	 \end{equation}
 	Note that given $E\in\mathbb{R}$, the space of solutions to (\ref{eq_ev_line}) is two-dimensional. In addition, given a solution $u$, its restriction to $\mathbb{Z}_\pm$ can be viewed as a half-line solution in the sense of Section \ref{section_hl_ops}.
 	\begin{definition}
 		Given $E\in\mathbb{R}$, a solution to (\ref{eq_ev_line}) is called \textit{subordinate} if and only if its restrictions to $\mathbb{Z}_\pm$ are subordinate (under suitable normalization) as half-line solutions.
 	\end{definition}
 	We will use the following theorem, which is a version of subordinacy theory line operators. 
	
	\begin{proposition}\emph{\cite[Theorem 2.7.11]{DF}}\label{sub_thm_line}
		Let $\mu_s$, $\mu_{ac}$ be the singular and absolutely continuous parts (with respect to the Lebesgue measure) of $\mu$ respectively. Then
		\begin{enumerate}
			\item $\mu_s$ is supported on the set
			\begin{center}
				$S=\left\{E\in\mathbb{R}:\text{(\ref{eq_ev_line}) has a subordinate solution}\right\}$.
			\end{center}
			\item $\mu_{ac}$ is supported on the set $N=N_-\cup N_+$, where
			\begin{center}
				$N_\pm=\left\{E\in\mathbb{R}:\text{ (\ref{eq_ev_line}) has no subordinate solution at }\pm\infty\right\}$.
			\end{center}
		\end{enumerate}
	\end{proposition}
	By the theory of rank-one perturbations and the Jitomirskaya-Last theory, if a subordinate solution exists, then it is unique (up to scalar multiplication). Thus, for every $E\in S$, there exists a unique $\theta\left(E\right)\in\left[0,\pi\right)$ such there exists subordinate solution $u$ of (\ref{eq_ev_line}) satisfying
	\begin{center}
		$\left(\begin{matrix}
			u\left(1\right)\\
			u\left(0\right)
		\end{matrix}\right)=\left(\begin{matrix}
			\cos\theta\\-\sin\theta
		\end{matrix}\right)$.
	\end{center}
	Let us define $H_\pm:\ell^2\left(\mathbb{Z}_\pm\right)\to\ell^2\left(\mathbb{Z}_\pm\right)$ by
	\begin{center}
		$H_\pm=P_{\mathbb{Z}_\pm} HP_{\mathbb{Z}_\pm}$,
	\end{center}
	and for every $\theta\in\left[0,\pi\right)$, let $H_+^\theta=H_+-\tan\theta\langle\delta_1,\cdot\rangle\delta_1$, and $H_-^\theta=H_--\cot\theta\langle\delta_0,\cdot\rangle\delta_0$. Finally, let $\mu_+^\theta$ and $\mu_-^\theta$ be the spectral measure of $\delta_1$ and $\delta_0$ with respect to $H_+^\theta$ and $H_-^\theta$ respectively, and let $m_\pm^\theta$ be their corresponding Borel transforms.
	We will use the following result from \cite{DKL}:
	\begin{proposition}\emph{\cite[Corollary 21]{DKL}}\label{DKL_Cor}
		Let $M:\mathbb{C}_+\to\mathbb{C}_+$ be the Borel transform of $\mu=\mu_0+\mu_1$. Then for every $z\in\mathbb{C}_+$,
		\begin{center}
			$\left|M\left(z\right)\right|\leq\underset{\theta\in\left[0,\pi\right)}{\sup}\left|m_+^\theta\left(z\right)\right|$.
		\end{center}
	\end{proposition}
	\subsection{Proof of Theorem \ref{main_thm_2}}
	Our goal in this section is to prove Theorem \ref{main_thm_2}. Namely, We will construct a Schr\"{o}dinger operator $H$ which satisfies the following properties:
	\begin{enumerate}
		\item For every $\theta\in\left[0,\pi\right)$, the essential spectrum of $H_{\pm}^\theta$ consists of the interval $\left[-2,2\right]$ and the restriction $\mu_\pm^\theta\left(\cdot\cap\left[-2,2\right]\right)$ is zero-dimensional.
		\item $H$ has one-dimensional spectrum, in the sense that for every Borel set $A$ with $\dim_H\left(A\right)<1$, we have $\mu\left(A\right)=0$.
	\end{enumerate}
	Before constructing $H$, let us shortly describe the idea. We will define a sequence of potentials, $\left(V^{\left(k\right)}\right)_{k=1}^\infty$, such that for every $n\in\mathbb{Z}$, the sequence $\left(V^{\left(k\right)}\left(n\right)\right)_{k=1}^\infty$ is eventually constant. Then, we will define $V$ by setting $V\left(n\right)=\underset{k\to\infty}{\lim}\,V^{\left(k\right)}\left(n\right)$.
	
	We start with $V^0\equiv 0$. By Corollary \ref{bdd_m_cor}, for every $\alpha\in\left(0,1\right)$, if the potential vanishes for a sufficiently long interval of the form $\left[-n,n\right]$, then for every $\theta\in\left[0,\pi\right)$, both of $\varepsilon^\alpha \left|m_{\pm}^\theta\left(E+i\varepsilon\right)\right|$ will be bounded. Now, if we change the value of the potential at some positive point, then obviously, only the positive half-line restriction of $H$ will be affected. Again, by Corollary \ref{bdd_m_cor}, we can now set the potential to be $0$ (both left and right to the origin) for a long enough time so that both of the left and right side $m$-functions will recover. Then, we may change the value of the potential at the negative half-line. If we continue with this process, we make sure that at each time (i.e.\ for every $\varepsilon>0$) one of the half-line $m$-functions will not be too big. On the other hand, as long as we set the values of the potential so that (\ref{eq_potential_growth}) holds, the resulting half-line spectral measures will have zero Hausdorff dimension.
	
	We will denote the (whole-line) potential and operator at each step by $V^{\left(k\right)}$ and $H^{\left(k\right)}$. We will also denote the half-line restrictions of $V$ by $V^{\left(k\right)}_{\pm}$. Let us now formally present the construction.
	
	\begin{itemize}
		\item We set $V^{\left(0\right)}\equiv 0$.
		\item By Corollary \ref{bdd_m_cor}, there exist $K\in\mathbb{N}$ and $\varepsilon_1>0$ such that for every potential $V$ which satisfies $V\left(-K\right)=\ldots=V\left(0\right)=\ldots=V\left(K\right)=0$, for every $\delta\in\left(\frac{\varepsilon_1}{2},\varepsilon_1\right)$, every $E\in\left[-2,2\right]$ and every $\theta\in\left[0,\pi\right)$,
		\begin{center}
			$\delta^{\alpha_1}\left|m_\pm^\theta\left(E+i\delta\right)\right|\leq1$,
		\end{center}
		where $\alpha_1=\frac{1}{2}$. We set $L_1^+=K+1$ and $V^1\left(j\right)=\begin{cases}
			M & j=L_1^+\\
			0 & \text{otherwise}
		\end{cases}$,
		where $M>0$ is chosen such that (\ref{eq_potential_growth}) holds for $n=1$.
		\item $V^{\left(1\right)}_\pm$ both eventually vanish and so by Corollary \ref{bdd_m_cor}, there exists $K\in\mathbb{N}$ and $\varepsilon_2<\frac{\varepsilon_1}{2}$ such that if $V\left(-1\right)=\ldots=V\left(-K\right)=0$ and $V\left(L_1^++1\right)=\ldots=V\left(L_1^++K\right)=0$, then for every $E\in\left[-2,2\right]$ and every $\theta\in\left[0,\pi\right)$, for every $\delta\in\left[\varepsilon_2,\varepsilon_1\right)$ we have
		\begin{center}
			$\delta^{\alpha_1}\left|m_{1,-}^\theta\left(E+i\delta\right)\right|\leq1$
		\end{center}
		and for every $\delta\in\left[\frac{\varepsilon_2}{2},\varepsilon_2\right)$ we have
		\begin{center}
			$\delta^{\alpha_2}\left|m_{1,\pm}^\theta\left(E+i\delta\right)\right|\leq1$,
		\end{center}
		where $\alpha_2=\frac{1}{3}$. We set $L_1^-=-K-1$ and $V^{\left(2\right)}\left(j\right)=\begin{cases}
			1 & j=L_1^-\\
			V^{\left(1\right)}\left(j\right) & \text{otherwise}
		\end{cases}$.
		\item $V^{\left(2\right)}_\pm$ both eventually vanish and so by Corollary \ref{bdd_m_cor} there exist $K\in\mathbb{N}$ and $\varepsilon_3<\frac{\varepsilon_2}{2}$ such that if $V\left(L_1^--1\right)=\ldots=V\left(L_1^--K\right)=0$ and $V\left(L_1^+\right)=V\left(L_1^++1\right)=\ldots=V\left(L_1^++K\right)=0$, then for every $E\in\left[-2,2\right]$ and every $\theta\in\left[0,\pi\right)$, for every $\delta\in\left[\varepsilon_3,\varepsilon_2\right)$ we have
		\begin{center}
			$\delta^{\alpha_2}\left|m_{2,+}^\theta\left(E+i\delta\right)\right|\leq1$
		\end{center}
		and for every $\delta\in\left[\frac{\varepsilon_3}{2},\varepsilon_3\right)$ we have
		\begin{center}
			$\delta^{\alpha_3}\left|m_{2,\pm}^\theta\left(E+i\delta\right)\right|\leq1$,
		\end{center}
		where $\alpha_3=\frac{1}{4}$. We set $L_2^+=L_1^++K+1$ and $V^{\left(3\right)}\left(j\right)=\begin{cases}
			A & j=L_1^+\\
			V^{\left(2\right)}\left(j\right) & \text{otherwise}
		\end{cases}$,
		where $A$ is some value such that the $4$-tuple $L_1^+,L_2^+,V^{\left(3\right)}\left(L_1^+\right),V^{\left(3\right)}\left(L_2^+\right)=A$ satisfies (\ref{eq_potential_growth}).
		\item $V^{\left(3\right)}_{\pm}$ both eventually vanish and so by Corollary \ref{bdd_m_cor} there exist $K\in\mathbb{N}$ and $\varepsilon_4<\frac{\varepsilon_3}{2}$ such that if $V\left(L_1^--1\right)=\ldots=V\left(L_1^--K\right)=0$ and $V\left(L_2^++1\right)=\ldots=V\left(L_2^++K\right)=0$, then for every $E\in\left[-2,2\right]$ and every $\theta\in\left[0,\pi\right)$, for every $\delta\in\left[\varepsilon_4,\varepsilon_3\right)$ we have
		\begin{center}
			$\delta^{\alpha_3}\left|m_{3,-}^\theta\left(E+i\delta\right)\right|\leq1$
		\end{center}
		and for every $\delta\in\left[\frac{\varepsilon_4}{2},\varepsilon_4\right)$ we have
		\begin{center}
			$\delta^{\alpha_4}\left|m_{3,\pm}^\theta\left(E+i\delta\right)\right|\leq1$,
		\end{center}
		where $\alpha_4=\frac{1}{5}$. We set $L_2^-=L_1^--K-1$ and $V^{\left(4\right)}\left(j\right)=\begin{cases}
			A & j=L_2^-\\
			V^{\left(3\right)}\left(j\right) & \text{otherwise}
		\end{cases}$, where $A$ is some value such that the $4$-tuple $L_1^-,L_2^-,V^{\left(4\right)}\left(L_1^-\right),V^{\left(4\right)}\left(L_2^-\right)=A$ satisfies (\ref{eq_potential_growth}).
		\item Suppose $V^{\left(2k\right)}$ was defined for some $k\in\mathbb{N}$. $V^{\left(2k\right)}_{\pm}$ both eventually vanish and so by Corollary \ref{bdd_m_cor}, there exist $K\in\mathbb{N}$ and $\varepsilon_{2k+1}<\frac{\varepsilon_{2k}}{2}$ such that if $V\left(L_k^--1\right)=\ldots=V\left(L_k^--K\right)=0$ and $V\left(L_k^++1\right)=\ldots=V\left(L_k^++K\right)=0$, then for every $E\in\left[-2,2\right]$ and every $\theta\in\left[0,\pi\right)$, for every $\delta\in\left[\varepsilon_{2k+1},\varepsilon_{2k}\right)$ we have
		\begin{center}
			$\delta^{\alpha_{2k}}\left|m_{2k,+}^\theta\left(E+i\delta\right)\right|\leq1$
		\end{center}
		and for every $\delta\in\left[\frac{\varepsilon_{2k+1}}{2},\varepsilon_{2k+1}\right)$ we have
		\begin{center}
			$\delta^{\alpha_{2k+1}}\left|m_{2k,\pm}^\theta\left(E+i\delta\right)\right|\leq1$,
		\end{center}
		where $\alpha_{2k+1}=\frac{1}{2k+2}$. We set $L_{k+1}^+=L_k^++K+1$ and \mbox{$V^{\left(2k+1\right)}\left(j\right)=\begin{cases}
				A & j=L_{k+1}^+\\
				V^{\left(2k\right)}\left(j\right) & \text{otherwise}
			\end{cases}$}, where $A$ is some value such that the $\left(n+1\right)$-tuple $L_1^+,\ldots,L_{k+1}^+,V^{\left(2k+1\right)}\left(L_1^+\right),\ldots,V^{\left(2k+1\right)}\left(L_{k+1}^+\right)$ satisfies (\ref{eq_potential_growth}).
		\item $V^{\left(2k+1\right)}_\pm$ are both eventually free and so by Corollary \ref{bdd_m_cor}, there exist $K\in\mathbb{N}$ and \mbox{$\varepsilon_{2k+2}<\frac{\varepsilon_{2k+1}}{2}$} such that if $V\left(L_k^--1\right)=\ldots=V\left(L_k^--K\right)=0$ and \mbox{$V\left(L_{k+1}^++1\right)=\ldots=V\left(L_{k+1}^++K\right)=0$}, then for every $E\in\left[-2,2\right]$ and every $\theta\in\left[0,\pi\right)$, for every $\delta\in\left[\varepsilon_{2k+2},\varepsilon_{2k+1}\right)$ we have
		\begin{center}
			$\delta^{\alpha_{2k+1}}\left|m_{2k+1,+}^\theta\left(E+i\delta\right)\right|\leq1$
		\end{center}
		and for every $\delta\in\left[\frac{\varepsilon_{2k+2}}{2},\varepsilon_{2k+2}\right)$ we have
		\begin{center}
			$\delta^{\alpha_{2k+1}}\left|m_{2k+1,\pm}^\theta\left(E+i\delta\right)\right|\leq1$,
		\end{center}
		where $\alpha_{2k+2}=\frac{1}{2k+3}$. We set $L_{k+1}^-=L_k^--K-1$ and \mbox{$V^{\left(2k+2\right)}\left(j\right)=\begin{cases}
				A & j=L_{k+1}^-\\
				V^{\left(2k+1\right)}\left(j\right) & \text{otherwise}
			\end{cases}$}, where $A$ is some value such that the $\left(n+1\right)$-tuple $L_1^-,\ldots,L_{k+1}^-,V^{\left(2k+2\right)}\left(L_1^-\right),\ldots,V^{\left(2k+2\right)}\left(L_{k+1}^-\right)=A$ satisfies (\ref{eq_potential_growth}).
	\end{itemize}
	Note that for every $n\in\mathbb{Z}$, if for some $k\in\mathbb{N}$, $V^k\left(n\right)\neq V^{k+1}\left(n\right)$, then for every $j>k$, \mbox{$V_j\left(n\right)=V^{k+1}\left(n\right)$}. Thus, we can define our potential $V$ by setting $V\left(n\right)=\underset{k\to\infty}{\lim}V^k\left(n\right)$ for every $n\in\mathbb{Z}$.
	\begin{proof}[Proof of Theorem \ref{main_thm_2}]
		We have defined the potential $V$ so that (\ref{eq_potential_growth}) will hold at $\pm\infty$, and so Property $1$ follows immediately by Theorem \ref{thm_zero_hausdorff_dim}. In addition, by our construction, for every $\alpha\in\left(0,1\right)$ there exists $\delta>0$ such that for every $\gamma\in\left(0,\delta\right)$ and $E\in\left[-2,2\right]$, either
		\begin{center}
			$\underset{\theta\in\left[0,\pi\right)}{\sup}\,\gamma^{1-\alpha}\left|m_-^\theta\left(E+i\gamma\right)\right|\leq 1$
		\end{center}
		or
		\begin{center}
			$\underset{\theta\in\left[0,\pi\right)}{\sup}\,\gamma^{1-\alpha}\left|m_+^\theta\left(E+i\gamma\right)\right|\leq 1$.
		\end{center}
		By Proposition \ref{DKL_Cor}, this implies that $\underset{\varepsilon\to0}{\limsup}\,\epsilon^{1-\alpha}\left|M\left(E+i\varepsilon\right)\right|<\infty$. Thus, by Propositions \ref{prop_hp_decomposition} and \ref{DJLS_thm}, Property $2$ holds.
	\end{proof}
	\subsection{Proof of Theorem \ref{main_thm_3}}
	 In the rest of this section we will prove Theorem \ref{main_thm_3}. We will prove a slightly more general result.
	 \begin{theorem}\label{main_thm_3_prelim}
	 	Let $H:D\left(H\right)\to\ell^2\left(\mathbb{Z}\right)$ be a Schr\"{o}dinger operator. Let $\mu,\left(\mu_\theta^\pm\right)_{\theta\in\left[0,\pi\right)}$ be defined as before. Suppose that there exist $\alpha,\beta\in\left[0,1\right]$ such that $\alpha<\beta$, $\dim_\text{H}\left(\mu\right)\geq\beta$ and for every $\theta\in\left[0,\pi\right)$, $\dim_\text{H}\left(\mu_\theta^\pm\right)\leq\alpha$. Then there exists a Borel set $A\subseteq\mathbb{R}$ such that $\mu\left(A\right)>0$, and one of the following must hold:
	 	\begin{enumerate}
	 		\item For every $\theta\in\left[0,\pi\right)$, $\mu_\theta^+\left(A\right)=0$, or
	 		\item For every $\theta\in\left[0,\pi\right)$, $\mu_\theta^-\left(A\right)=0$.
	 	\end{enumerate}
	 \end{theorem}
 	We begin with the following
	\begin{lemma}\label{assist_claim}
		Suppose $\mathcal{H}=\mathcal{H}_1\oplus\mathcal{H}_2$, $T_i:D\left(T_i\right)\subseteq\mathcal{H}_i\to\mathcal{H}_i$ are self-adjoint operators with cyclic vectors $\varphi_i$, and $T:D\left(T\right)\subseteq\mathcal{H}\to\mathcal{H}$ is given by $T=T_1\oplus T_2+\langle\varphi,\cdot\rangle\varphi$, where $\varphi=\varphi_1\oplus\varphi_2$. Denote by $\mu_i$ the spectral measure of $\varphi_i$ with respect to $T_i$ and by $\mu$ the sum of spectral measures of $\varphi_i$ w.r.t.\ $T$. Let $X,Y,Z\subseteq\mathbb{R}$ be Borel sets such that $\mu_1|_X\sim\mu_2|_X$ (in  the sense of mutual absolute continuity) and $\mu_1\left(Y\right)=\mu_2\left(Z\right)=0$. Then, if $\mu_1\left(X\right)>0$ then $\mu\left(X\right)>0$.
	\end{lemma}
	\begin{proof}
		Denote $B=T_1\oplus T_2$. Let $\mathcal{K}_1$ be the cyclic subspace of $\mathcal{H}$ generated by $B$ and $\varphi$, and let $\mathcal{K}_2=\mathcal{K}_1^\perp$. Note that in $L^2\left(\mathbb{R},d\mu_1\right)\oplus L^2\left(\mathbb{R},d\mu_2\right)$, we have that $\left(h_1,h_2\right)\in\mathcal{K}_2$ if and only if $h_j$ are supported on $X$, and for every bounded Borel function $g$,
		\begin{center}
			$\int gh_1\,d\mu_1+\int gh_2\,d\mu_2=0$.
		\end{center}
		This implies that $\mathcal{K}_2\neq\left\{0\right\}$ if and only if $\mu_1\left(X\right)>0$, and that the spectral measure (with respect to $B$) of vectors in $\mathcal{K}_2$ is supported on $X$. Finally, note that for every $\psi\in\mathcal{K}_2$, we have $\langle\varphi,\psi\rangle=0$ and so $H\psi=B\psi$. Thus, $H|_{\mathcal{K}_2}=B|_{\mathcal{K}_2}$. Now, suppose that $\mu_1\left(X\right)>0$ and take any non-zero vector $\psi\in\mathcal{K}_2$. Then its spectral measure (with respect to $B$) is supported on $X$, and so its spectral measure (with respect to $H$) is also supported on $X$. This implies that $\mu\left(X\right)>0$, as required.
	\end{proof}
	Let us now consider $H:D\left(H\right)\subseteq\ell^2\left(\mathbb{Z}\right)\to\ell^2\left(\mathbb{Z}\right)$ which satisfies the assumptions of Theorem \ref{main_thm_3_prelim}. Denote $\widetilde{I}=\left[0,\pi\right)\setminus\left\{0,\frac{\pi}{2}\right\}$. For every $\theta\in\widetilde{I}$, $H$ can be written as \mbox{$H=H_-^{\theta}\oplus H_+^\theta+\langle\varphi_\theta,\cdot\rangle\varphi_\theta$} for some \mbox{$\varphi_\theta=\varphi_-^\theta\oplus\varphi_+^\theta$}, where $\varphi_\pm^\theta$ is a cyclic vector for $H_\pm^\theta$. This can be seen by noting that, if we denote $B=H_-^{\theta}\oplus H_+^\theta$, then for every $\psi\in\ell^2\left(\mathbb{Z}\right)$, we have
	\begin{center}
		$\left(H-B\right)\psi\left(n\right)=\begin{cases}
			0 & n\notin\left\{0,1\right\}\\
			\psi\left(0\right)-\tan\theta\psi\left(1\right) & n=1\\
			\psi\left(1\right)-\cot\theta\psi\left(0\right) & n=0
		\end{cases}$
	\end{center}
	and now, it is easy to verify that $\varphi_\theta=\sqrt{-\cot\theta}\delta_0+\sqrt{-\tan\theta}\delta_1$ will do the job.
	
	By Proposition \ref{sub_thm_line}, the spectral measure of $H$, $\mu$, is supported on the set of energies $E\in\mathbb{R}$ for which there exists a solution $\psi$ to (\ref{eq_ev_line}) which is subordinate at $\pm\infty$. We will denote this set by $S$. Recall that $\mu=\mu_{\delta_0}+\mu_{\delta_1}$. Let $\widetilde{S}\subseteq S$ be the set of energies for which the subordinate solution does not vanish at $0$ and at $1$. For every $i,j\in\mathbb{Z}$, let us denote by $\mu_{ij}$ the spectral measure of $\delta_i$ and $\delta_j$ with respect to $H$, and by $M_{ij}$ its Borel transform. It is shown in the appendix of \cite{L} that for every $E\in \widetilde{S}$, the corresponding subordinate solution $\psi$ satisfies
	\begin{center}
		$\psi\left(j\right)=\underset{\varepsilon\to0}{\lim}\frac{M_{j1}\left(E+i\varepsilon\right)}{M_1\left(E+i\varepsilon\right)},\,\,\,\,\,j\in\left\{0,1\right\}$
	\end{center}
	and so
	\begin{center}
		$\frac{\psi\left(0\right)}{\psi\left(1\right)}=\underset{\varepsilon\to0}{\lim}\frac{M_{01}\left(E+i\varepsilon\right)}{M_1\left(E+i\varepsilon\right)}=\frac{d\mu_{01}}{d\mu_1}\left(E\right)$.
	\end{center}
	In particular, since $\theta\left(E\right)=-\arctan\left(\frac{\psi\left(1\right)}{\psi\left(0\right)}\right)$, we get that $E\to\theta\left(E\right)$ is a measurable function. In addition, the function $G:\mathbb{R}\times\mathbb{C}_+\to\mathbb{R}$ which is defined by $G\left(\theta,z\right)=\frac{\im m_+^\theta\left(z\right)}{\im m_+^\theta\left(z\right)+\im m_-^{\theta}\left(z\right)}$ is continuous. Let $\left(q_n\right)_n\in\mathbb{N}$ be an enumeration of $\mathbb{Q}\cap\left(0,1\right)$. For every $n\in\mathbb{N}$, the function $F_n:\mathbb{R}\to\mathbb{R}$ which is defined by \mbox{$F_n\left(E\right)=G\left(\theta\left(E\right),E+iq_n\right)$} is measurable. Denote
	\begin{center}
		$A_1=\left\{E\in\mathbb{R}:\underset{n\to\infty}{\liminf}\,F_n\left(E\right)=0\right\},\,\,\,\,A_2=\mathbb{R}\setminus A_1$.
	\end{center} 
	Finally, denote $\widetilde{A_i}=\widetilde{S}\cap A_i$ for $i\in\left\{1,2\right\}$.
	\begin{claim}
		Either $\mu\left(\widetilde{A_1}\right)>0$ or $\mu\left(\widetilde{A_2}\right)>0$.
	\end{claim}
	\begin{proof}
		It suffices to show that $\mu\left(\widetilde{S}\right)>0$. Suppose not, and for $i=0,1$, let $N_i$ be the set of $E\in\mathbb{R}$ for which the subordinate solution vanishes at $i$. Recall that $P$ is the projection-valued spectral measure of $H$. By our assumption, $P\left(\widetilde{S}\right)=0$ and so \mbox{$\ell^2\left(\mathbb{Z}\right)=\Ran P\left(N_0\right)\oplus\Ran P\left(N_1\right)\coloneqq\mathcal{H}_0\oplus\mathcal{H}_1$}. Furthermore, every $\varphi\in\mathcal{H}_i$ vanishes at $\delta_i$. Thus, it must be the case that $\delta_0\in\mathcal{H}_0$. In addition, $\mathcal{H}_0$ is invariant under $H$. On the other hand, clearly, $\left(H\delta_0\right)\left(1\right)\neq 0$, which is a contradiction.
	\end{proof}
	\begin{proposition}\label{any_bd_con_prop}
		For every $\theta\in\widetilde{I}$, $\mu_{+}^\theta\left(\widetilde{A_1}\right)=\mu_-^{\theta}\left(\widetilde{A_2}\right)=0$.
	\end{proposition}
	\begin{proof}
		First note that by Proposition \ref{prop_JL}, $\mu_+^\theta\left(\widetilde{A_1}\right)=\mu_+^\theta\left(B\right)$, where \mbox{$B=\left\{E\in\widetilde{A_1}:\theta\left(E\right)=\theta\right\}$}. Note that for every $E\in B$, we have $\underset{n\to\infty}{\liminf}\frac{\im m_+^\theta\left(E+q_n\right)}{\im m_+^\theta\left(E+q_n\right)+\im m_-^{\theta}\left(E+q_n\right)}=0$. In particular, $\underset{\varepsilon\to0}{\liminf}\,\frac{\im m_+^\theta\left(E+i\varepsilon\right)}{\im m_+^\theta\left(E+i\varepsilon\right)+\im m_-^{\theta}\left(E+i\varepsilon\right)}=0$. By Proposition \ref{rnd_m_func_thm}, we know that
		\begin{equation}\label{eq_g_def}
		    g\coloneqq\frac{d\mu_+^\theta}{d\left(\mu_+^\theta+\mu_-^{\theta}\right)}=\underset{\varepsilon\to0}{\lim}\frac{\im m_+^\theta\left(E+i\varepsilon\right)}{\im m_+^\theta\left(E+i\varepsilon\right)+\im m_-^{\theta}\left(E+i\varepsilon\right)}
		\end{equation}
		and so we get
		\begin{center}
			$\mu_+^\theta\left(B\right)=\int_Bgd\left(\mu_+^\theta+\mu_-^{\theta}\right)=0$,
		\end{center}
		as required. To see that $\mu_-^{\theta}\left(\widetilde{A_2}\right)=0$, first note that we cannot have a Borel set $X\subseteq\mathbb{R}$ such that $\mu_+^\theta|_X\sim\mu_-^\theta|_X$. Indeed, if it were the case, then by Lemma \ref{assist_claim} we would have that $\mu\left(X\right)>0$. On the other hand, since both $\mu_\pm^\theta$ are of Hausdorff dimension $\leq\alpha$, we may assume that $\dim_\text{H}\left(X\right)\leq\alpha$. Finally, since $\mu$ assigns zero weight to every set of Hausdorff dimension less than $\beta$, we obtain $\mu\left(X\right)=0$ which contradicts our assumption. This implies that for any decomposition of $\mathbb{R}$ to disjoint Borel sets $\mathbb{R}=X\sqcup Y\sqcup Z$ such that $\mu_+^\theta\left(Y\right)=\mu_-^{\theta}\left(Z\right)=0$ and $\mu_+^\theta|_X\sim\mu_-^{\theta}|_X$, we must have $\mu_+^\theta\left(X\right)=\mu_-^{\theta}\left(X\right)=0$. Note that such a decomposition can be taken to be
		\begin{center}
			$X=\left\{E\in\mathbb{R}:g\left(E\right)\in\left(0,1\right)\right\}$,\\
			$Y=\left\{E\in\mathbb{R}:g\left(E\right)=0\right\}$,\\
			$Z=\left\{E\in\mathbb{R}:g\left(E\right)=1\right\}$,
		\end{center}
		where $g$ is given by (\ref{eq_g_def}). Thus, we have that
		\begin{center}
			$\mu_-^{\theta}\left(\widetilde{A_2}\right)=\mu_-^{\theta}\left(\widetilde{A_2}\cap X\right)+\underset{=0}{\underbrace{\mu_-^{\theta}\left(\widetilde{A_2}\cap Z\right)}}=\mu_-^{\theta}\left(\widetilde{A_2}\cap X\right)=\mu_-^{\theta}\left(\emptyset\right)=0$,
		\end{center} 
		as required.
	\end{proof}
	Finally, we are ready to prove Theorem \ref{main_thm_3_prelim}, which will immediately imply Theorem \ref{main_thm_3}.
	\begin{proof}[Proof of Theorem \ref{main_thm_3_prelim}]
		If $\mu\left(\widetilde{A_1}\right)>0$, then by Proposition \ref{any_bd_con_prop} we have that $\mu_+^\theta\left(\widetilde{A_1}\right)=0$ for every $\theta\in\widetilde{I}$. If $\mu\left(\widetilde{A_2}\right)>0$ then for every $\theta\in\widetilde{I}$, $\mu_-^{\theta}\left(\widetilde{A_2}\right)=0$. Now, note that for $\theta=0,\frac{\pi}{2}$, $\mu_\pm^\theta$ is supported on the set of $E\in\mathbb{R}$ for which $\theta\left(E\right)=0,\frac{\pi}{2}$ respectively, and clearly, for every $E\in\widetilde{S}$, $\theta\left(E\right)\neq 0,\frac{\pi}{2}$. Thus, for $\theta=0,\frac{\pi}{2}$, $0=\mu_\pm^\theta\left(\widetilde{S}\right)\geq\mu_{\pm}^\theta\left(\widetilde{A_i}\right)$ and Theorem \ref{main_thm_3_prelim} is proved and consequently, Theorem \ref{main_thm_3} is also proved.
	\end{proof}
	

\begin{thebibliography}{10}
		
		\bibitem{Ber} Ju.~M.~Berezanski\u{i}, {\it Expansions in Eigenfunctions of Self-Adjoint Operators}, Transl.\ Math.\ Mono. \textbf{17}, Amer.\ Math.\ Soc., Providence, RI, 1968.
		
		\bibitem{CdO} C.~L.~Carvalho and C.~R.~de Oliveira, {\it Spectral packing dimensions through power-law subordinacy}, Ann.\ Henri Poincar\'{e} \textbf{14} (2013), 775--792.
		
		\bibitem{Cut} C.~D.~Cutler, {\it  Measure disintegrations with respect to $\sigma$-stable monotone indices and the pointwise representation of packing dimension}, Supp.\ Rend.\ Circ.\ Mat.\ Palermo \textbf{28} (1992), 319--339.
		
		\bibitem{DF} D.~Damanik and J.~Fillman, {\it One-dimensional ergodic Schrödinger operators--I. General theory}, Grad.\ Stud.\ Math.\ \textbf{221}, American Mathematical Society, Providence, RI (2022).

        \bibitem{DF2} D.~Damanik and J.~Fillman, {\it One-dimensional ergodic Schrödinger operators--II. Specific classes}, Grad.\ Stud.\ Math.\ \textbf{221}, American Mathematical Society, Providence, RI (2025).
        
		\bibitem{DJLS} R.~del Rio, S.~Jitomirskaya, Y.~Last and B.~Simon, {\it Operators with singular continuous spectrum, IV. Hausdorff dimensions, rank one perturbations, and localization
		}, J. Anal. Math \textbf{69} (1996), 153--200.
	
		\bibitem{DJMS} R.~del Rio, S.~Jitomirskaya, N.~Makarov and B.~Simon, {\it Singular continuous spectrum is generic}, Bull.\ Amer.\ Math.\ Soc. \textbf{31} (1994), 208--212.
	
		\bibitem{DKL} D.~Damanik, R.~Killip and D.~Lenz, {\it Uniform Spectral Properties of One-Dimensional Quasicrystals, III. $\alpha$-continuity}, Commun.\ Math.\ Phys \textbf{212} (2000), 191--204.
		
		\bibitem{DMS} R.~del Rio, N.~Makarov and B.~Simon, {\it Operators with singular continuous spectrum. II. Rank one
			operators}, Commun.\ Math.\ Phys.\ \textbf{165} (1994), 59--67.
		
		\bibitem{DT} D.~Damanik and S.~Tcheremchantsev, {\it Scaling estimates for solutions and dynamical lower bounds on wavepacket spreading}, J.\ Anal.\ Math \textbf{97} (2005), 103--131.
		
		\bibitem{Gor} A.~Ya.~Gordon {\it Pure point spectrum under 1-parameter perturbations and instability of Anderson localization}, Commun.\ Math.\ Phys.\ \textbf{164} (1994), 489--505.
		
		\bibitem{GP} D.~J.~Gilbert and D.~B.~Pearson, {\it On subordinacy and analysis of the spectrum of one-dimensional Schr\"{o}dinger operators}, J.\ Math.\ Anal.\ Appl. \textbf{128} (1987), 30--56.
		
		\bibitem{GSB} I.~Guarneri and H.~Schulz-Baldes, {\it Lower bounds on wave packet propagation by packing dimensions of spectral measures}, Math.\ Phys.\ Electron.\ J.\ \textbf{16} (1999), Paper 1.
		
		\bibitem{JK} S.~Jitomirskaya and I.~Kachkovskiy, {\it Sharp arithmetic localization for quasiperiodic operators with monotone potentials.}, preprint. 	arXiv:2407.00703.
		
		\bibitem{JL1} S.~Jitomirskaya and Y.~Last, {\it Power law subordinacy and singular spectra, I. Half-line operators}, Acta Math. \textbf{183} (1999), 171--189.
		
		\bibitem{JL2} S.~Jitomirskaya and Y.~Last, {\it Power law subordinacy and singular spectra, II. Line operators}, Commun.\ Math.\ Phys. \textbf{211} (2000), 643--658.
		
		\bibitem{JLT} S.~Jitomirskaya, W.~Liu and S.~Tcheremchantsev, {\it Lower bounds on concentration through Borel transform and quantitative singularity of spectral measures near the arithmetic transition}, preprint. arXiv:2501.12153.
		
		\bibitem{Kac} I.~S.~Kac, {\it Spectral multiplicity of a second-order differential operator and expansion in eigenfunction. (Russian)}, Izv.\ Akad.\ Nauk SSSR Ser.\ Mat. \textbf{27} (1963), 1081--1112.
		
		\bibitem{Kato} T.~Kato, {\it Perturbation Theory for Linear Operators}, Springer-Verlag, Berlin, 1995.
		
		\bibitem{KKL} R.~Killip, A.~Kiselev and Y.~Last, {\it Dynamical upper bounds on wavepacket spreading}, Amer.\ J.\ Math.\ \textbf{125} (2003), 1165--1198.
		
		\bibitem{KP} S.~Khan and D.~B.~Pearson, {\it Subordinacy and spectral theory for infinite matrices}, Helv.\ Phys.\ Acta \textbf{65} (1992), 505--527.
		
		\bibitem{KPS} I.~Kachkovskiy, L.~Pranovski and R.~Shterenberg, {\it On gaps in the spectra of quasiperiodic Schrödinger operators with discontinuous monotone potentials}, preprint. arXiv:2407.00705.
		
		\bibitem{L} N.~Levi, {\it Eigenfunction Expansion and the Decomposition of Jacobi Operators on $\mathbb{Z}$}, preprint. arXiv:2502.03707.
		
		\bibitem{LS} Y.~Last and B.~Simon, {\it Eigenfunctions, transfer matrices, and absolutely continuous spectrum of one-dimensional Schr\"{o}dinger operators}, Invent.\ Math.\ \textbf{135} (1999), 329--367.
		
		\bibitem{RS} M.~Reed and B.~Simon, {\it Methods of modern mathematical physics, I: Functional analysis}, Academic press, New York, 1972.
		
		\bibitem{RT} C.~A.~Rogers and S.~J.~Taylor, {\it The analysis of additive set functions in Euclidean space}, Acta Math. \textbf{101} (1959), 273--302.
		
		\bibitem{Z} A.~Zlato\u{s}, {\it Sparse potentials with fractional Hausdorff dimension}, J.\ Funct.\ Anal.\ \textbf{207} (2004), 216--252.
		
	\end{thebibliography}
\end{document}